\documentclass[11pt]{amsart}

\newtheorem{thm}{Theorem}[section]

\newtheorem{cor}[thm]{Corollary}

\newtheorem{Def}[thm]{Definition}
\newtheorem{prop}[thm]{Proposition}
\newtheorem{rem}[thm]{Remark}
\newtheorem{ex}[thm]{Example}

\newcommand{\bdfn}{\begin{Def} \rm}
	\newcommand{\edfn}{\end{Def}}

\newcommand{\ra}{\rightarrow}
\newcommand{\Ra}{\Rightarrow}

\newcommand{\es}{\emptyset}
\newcommand{\ci}{\subseteq}
\newcommand{\ds}{\displaystyle}

\newcommand{\al}{\alpha}
\newcommand{\be}{\beta}

\newcommand{\e}{\varepsilon}

\newcommand{\la}{\lambda}

\newcommand{\ga}{\gamma}
\newcommand{\Ga}{\Gamma}

\newcommand{\mb}{\mathbb}
\newcommand{\mc}{\mathcal}

\newcommand{\sm}{\setminus}

\newcommand{\iy}{\infty}

\newcommand{\beqa}{\begin{eqnarray*}}
	\newcommand{\eeqa}{\end{eqnarray*}}

\newcounter{cnt1}
\newcounter{cnt2}
\newcounter{cnt3}
\newcounter{cnt4}
\newcommand{\blr}{\begin{list}{$($\roman{cnt1}$)$} {\usecounter{cnt1}
			\setlength{\topsep}{0pt} \setlength{\itemsep}{0pt}}}
	\newcommand{\blR}{\begin{list}{\Roman{cnt4}.\ } {\usecounter{cnt4}
				\setlength{\topsep}{0pt} \setlength{\itemsep}{0pt}}}
		\newcommand{\bla}{\begin{list}{$(\alph{cnt2})$} {\usecounter{cnt2}
					\setlength{\topsep}{0pt} \setlength{\itemsep}{0pt}}}
			\newcommand{\bln}{\begin{list}{$($\arabic{cnt3}$)$} {\usecounter{cnt3}
						\setlength{\topsep}{0pt} \setlength{\itemsep}{0pt}}}
				\newcommand{\el}{\end{list}}

\begin{document}

\title[Uniqueness of Hahn-Banach extension]{Uniqueness of Hahn-Banach extension and related norm-$1$ projections in dual spaces}

\author[Daptari]{Soumitra Daptari}
\address{Department of Mathematics,	Indian Institute of Technology Hyderabad,
	Kandi Campus, Telangana 502285, India}
\email{ma17resch11003@iith.ac.in}

\author[Paul]{Tanmoy Paul}
\address{Department of Mathematics,	Indian Institute of Technology Hyderabad,
	Kandi Campus, Telangana 502285, India}
\email{tanmoy@math.iith.ac.in}

\author[Rao]{T. S. S. R. K. Rao}
\address{Department of Mathematics, Ashoka University,
	Rajiv Gandhi Education City, Sonipat, Haryana 131029, India}
\email{srin@fulbrightmail.org}



\begin{abstract}
In this paper we study two properties viz. property-$U$ and property-$SU$ of a subspace $Y$ of a Banach space which correspond to the uniqueness of the Hahn-Banach extension of each linear functional in $Y^*$ and in addition to that this association forms a linear operator of norm-1 from $Y^*$ to $X^*$.  It is proved that, under certain geometric assumptions on $X, Y, Z$ these properties are stable with respect to the injective tensor product; $Y$ has property-$U$ ($SU$) in $Z$ if and only if $X\otimes_\e^\vee Y$ has property-$U$ ($SU$) in $X\otimes_\e^\vee Z$.
We prove that when $X^*$ has the Radon-Nikod$\acute{y}$m Property for $1<p< \infty$, $L_p(\mu, Y)$ has property-$U$ (property-$SU$) in $L_p(\mu, X)$ if and only if $Y$ is so in $X$. We show that if $Z\ci Y\ci X$, where $Y$ has property-$U$ ($SU$) in $X$ then $Y/Z$ has property-$U$ ($SU$) in $X/Z$. On the other hand $Y$ has property-$SU$ in $X$ if $Y/Z$ has property-$SU$ in $X/Z$ and $Z (\ci Y)$ is an M-ideal in $X$. It is observed that a smooth Banach space of dimension $>3$ is a Hilbert space if and only if for any two subspaces $Y, Z$ with property-$SU$ in $X$, $Y+Z$ has property-$SU$ in $X$ whenever $Y+Z$ is closed.
We characterize all hyperplanes in $c_0$ which have property-$SU$.
\end{abstract}

\maketitle

{\centering\footnotesize Dedicated to the memory of Eve Oja.\par}

\section{Introduction}

By $X$ we mean a Banach space with real scalars and a subspace $Y$ of $X$ is always assumed to be a closed subspace of $X$ except when we consider dense linear subspaces. The annihilator of $Y$ in $X^*$ is denoted by $Y^{\perp}$ and $Y^{\#}$ represents the set $\{f\in X^*:\|f\|=\|f|_Y\|\}$. $B_X$ and $S_X$ represent the closed unit ball and closed unit sphere of $X$ respectively. By a {\it hyperplane} in $X$ we mean a subspace of the form $f^{-1}(0)$, for some $f\in X^*$. The classical Hahn-Banach Theorem ensures that every $f\in Y^*$ has a norm preserving extension $\tilde{f}\in X^*$. Uniqueness of such extension depends on some geometric structures of $X^*$. The authors in \cite{F, P, T} studied this property and characterize those Banach spaces where a given subspace (all subspaces) has (have) this unique extension property. We now recall the following definitions.

\bdfn
\cite{P} A subspace $Y$ of $X$ said to have property-$U$ if every $f\in Y^*$ has unique norm preserving extension $\tilde{f}$ to $X^*$.
\edfn

Also let us recall the following definitions.

\bdfn
\bla
\item \cite{L} A subspace $Y$ of $X$ is said to be an ideal if there exists a projection $P:X^*\ra X^*$ where $\|P\|=1$ and $ker (P)=Y^\perp$.
\item \cite{O} A subspace $Y$ of $X$ is said to have property-$SU$ if $Y^\perp$ has a complementary subspace $G\ci X^*$ such that for all $f\in X^*$, $f=g+h$ $g\in G, h\in Y^\perp$ and $\|f\|>\|g\|$ whenever $h\neq 0$.
\el
\edfn

It is well known that there exists a linear operator $S: Y^*\ra X^*$ where $S(y^*)$ is a norm preserving extension of $y^*$ if and only if $Y$ is an ideal in $X$. Also recall that,
$Y$ is an ideal in $X$ if and only if $Y^{\perp\perp}$ is a range of a norm-$1$ projection in $X^{**}$(\cite{Rao}).
It is clear from the properties of a subspace $Y$ having property-$SU$ is that the canonical restriction map from $Y^\#$ to $Y^*$ is an isometry. In \cite{O} Oja observed, a subspace $Y$ has property-$SU$ in $X$ if and only if $Y$ has property-$U$ and also an ideal in $X$.

Let us recall that a subspace $Y$ is said to have $n.X.I.P.$ in $X$ if $n$ closed balls $\{B(a_i,r_i)\}_{i=1}^n$ in $X$ with centres in $Y$ and $\bigcap_{i=1}^n B(a_i,r_i)\neq\es$, then $Y\cap \bigcap_{i=1}^n B(a_i,r_i+\e)\neq\es$, for all $\e>0$. It follows from \cite{L} that,

\begin{thm}\label{T1}
	Let $Y$ be a subspace of a Banach space $X$. If $Y$ is an ideal in $X$ then $Y$ has $n.X.I.P.$, for all $n\in\mb{N}$. Suppose $Y$ has property-$U$ in $X$, if $Y$ has $n.X. I. P.$ in $X$ then $Y$ is an ideal in $X$. Here it is enough to consider $n = 3$.
\end{thm}

In this paper we discuss property-$U$ and $SU$ for various kinds of function spaces, followed by their stability in  quotient spaces. We also study these properties for subspaces of Banach spaces which are of type $L_1$-preduals.
Various examples are given at the end of the paper to illustrate the limitations of some of the results obtained here using the sequence space $c_0$ and
finite dimensional spaces.

\section{Notations and Definitions}

Let $Y$ be a subspace of $X$ and let $x\in X$, define $d(x,Y)=\inf_{y\in Y}\|x-y\|$ and $P_Y(x)=\{y\in Y:\|x-y\|=d(x,Y)\}$. Note that $P_Y(x)$ may be empty for some $x$.

\bdfn
\bla
\item $Y$ is said to be {\it proximinal} if $P_Y(x)\neq\es$ for all $x$.
\item $Y$ is said to be {\it Chebyshev} if $P_Y(x)$ is singleton for all $x$.
\el
\edfn

\cite{IS} is a standard reference for the notions defined above.

\bdfn\cite{HW}
\bla
\item $Y$ is said to be an M-summand (L-summand) in $X$ if there exists a linear projection $P:X\ra X$ such that $Y=P(X)$ and $X=P(X)\bigoplus_{\ell_\iy} (I-P)(X)$ $\left(X=P(X)\bigoplus_{\ell_1} (I-P)(X)\right)$.
\item $Y$ is said to be an M-ideal in $X$ if $Y^\perp$ is an L-summand in $X^*$.
\item $Y$ is said to be a semi M-ideal in $X$ if for any $x^*\in X^*$, we have $\|x^*\|=\|Px^*\|+\|x^*-Px^*\|$, where $P:X^*\ra X^*$ is a projection satisfying $P(\la x_1^*+P(x_2^*))=\la P (x_1^*)+P (x_2^*)$, for $x_1^*, x_2^*\in X^*$ and $\la\in \mb{R}$ ({\it quasi additivity}) (see \cite[Pg. 43]{HW}).
\el
\edfn

Note that if $Y$ is an M-ideal in $X$ and $X^*=Y^\perp\bigoplus_{\ell_1} Z$ then $Z\cong Y^\#$, where $Y^\#$ is defined in Section~1.
A Banach space $X$ is said to be an {\it M-embedded space} if $X$ is an {\it M-ideal} in $X^{**}$ under the canonical embedding. An M-embedded space $X$ satisfies many geometric properties, in particular this property is separably determined and $X^*$ has {\it Radon-Nikod$\acute{y}$m property} (see below) \cite[Pg. 126, Theorem~3.1]{HW}. $c_0(\Ga)$, for an arbitrary set $\Ga$, $K(\ell_p)$, for $1<p<\iy$ are some examples of M-embedded spaces. If $X$ is M-embedded then so is all its subspaces.

Let $(\Omega, \Sigma, \nu)$ be a positive measure space, a measurable function $f:\Omega\ra X$ is said to be $p$-th Bochner-Integrable function if $\int_\Omega \|f(t)\|^pd\nu (t)<\iy$. The corresponding $p$-th norm is defined by $\|f\|_p=\left(\int_\Omega \|f(t)\|^pd\nu (t)\right)^\frac{1}{p}$, where $1\leq p<\iy$. The corresponding Banach space consisting of all $p$-th Bochner Integrable functions is denoted by $L_p(\nu, X)$. $L_\iy(\nu, X)$ represents the space of essentially bounded, measurable functions from $\Omega$ to $X$. $L_\iy (\nu, X)$ forms a Banach space with respect to the essentially supremum norm.

Let us recall the following Definition.

\bdfn\cite{DU}\label{D2}
A Banach space $X$ is said to have Radon-Nikod$\acute{y}$m property (RNP in short) if for any probability space $(\Omega, \Sigma, \mu)$ and any $\mu$-continuous vector measure $G:\Sigma\ra X$ of bounded variation there exists $g\in L_1(\mu, X)$ such that $G(E)=\int_E g(t)d\mu (t)$ for all $E\in\Sigma$.
\edfn

Numerous characterizations are available in the literature for this property. In some special cases the Banach spaces of vector valued functions can be expressed as the tensor product of classical spaces. The monographs \cite{DU, LC} are some standard references of tensor product of Banach spaces and related properties.
We follow the notations from \cite{DU}, to define the {\it injective} ({\it projective}) tensor product of two Banach spaces $X, Y$, which will be denoted by $X\otimes_\e^\vee Y$ ($X\otimes_\pi^\wedge Y$).

For $z\in X\otimes Y$ and let $\sum_{i=1}^nx_i\otimes y_i$ be one such representation of $z$.
Define the following {\it cross norms} on $X\otimes Y$,

\[
\la \left(\sum_{i=1}^n x_i\otimes y_i\right)=\sup \left\{\left\|\sum_{i=1}^n\phi (x_i)y_i\right\|:\phi\in S_{X^*}\right\} \mbox{ ~and also, }
\]
\[
\ga (z)= \inf \left\{\sum_{i=1}^n \|x_i\|\|y_i\|:z \approx \sum_{i=1}^n x_i\otimes y_i\right\}.
\]

The completion of $(X\otimes Y, \la)$ ($(X\otimes Y, \ga)$) in $\mc{B}(X\times Y)^*$ is said to be the injective (projective) tensor product of $X$ and $Y$ and is denoted by $X\otimes_\e^\vee Y$ ($X\otimes_\pi^\wedge Y$).

It is well-known that (see \cite[Pg. 114]{RR1}) if one of the spaces $X^*, Y^*$ has the RNP and one of them has {\it Approximation property} then $\left(X\otimes_\e^\vee Y\right)^*\cong X^*\otimes_\pi^\wedge Y^*$ and also $L_p(\nu, X)^*\cong L_q(\nu, X^*)$, where $\frac{1}{p}+\frac{1}{q}=1$, where $1\leq p<\iy$, $q=\iy$ when $p=1$. Also $L_1(\mu, X)\cong L_1(\mu)\otimes_\pi^\wedge X$ and $C(K, X)\cong C(K)\otimes_\e^\vee X$, where $X$ is a Banach space and $K$ is a compact Hausdorff space. Here $C(K,X):=\{f:K\ra X: f\mbox{~is continuous on~}K\}$ and the corresponding Banach space norm is $\|f\|_\iy=\sup_K\|f(k)\|$ (see \cite{LC}).

\bdfn
A Banach space $X$ is said to be an $L_1$-predual space if $X^*$ is isometrically isomorphic to $L_1(\mu)$ for some measure space $(\Omega,\Sigma,\mu)$.
\edfn

We refer Chapter~6 and 7 from the monograph \cite{HE} by Lacey for characterizations of these spaces and their properties. It is well-known that $C(K)$ is an $L_1$-predual where $K$ is compact Hausdorff. An $L_1$-predual space and its dual always have the approximation property.

If $B$ is a closed, bounded, convex  subset of a normed space, then $ext (B):=\{b\in B:b=\frac{x_1+x_2}{2} \mbox{~for some~}x_1, x_2\in B \mbox{~then~}b=x_1=x_2 \}$, the set of all {\it extreme points} of $B$.

\section{Results on Banach spaces of vector valued functions}

In this section we explore properties-$U$ and $SU$ in spaces of Bochner Integrable functions and Banach spaces of vector valued continuous functions. Both these spaces can be interpreted as tensor products over suitable Banach spaces. Theorem~\ref{T8}, \ref{T7} are the main observations in this section.
We first prove a formula for the quotient of the spaces of Bochner Integrable functions.

\begin{prop}\label{P5}
	Let $(\Omega, {\Sigma}, \mu)$ be a probability space. Let $X$ be a Banach space and $Y$ be a subspace of $X$. Fix $1\leq p<\iy$.
	Then $L_p(\mu, X)/L_p(\mu, Y)$ is isometric to $L_p(\mu, X/Y)$.
\end{prop}

\begin{proof}
	Let $\pi:X\ra X/Y$ denote the quotient map.
	
	Define $\Phi: L_p(\mu, X) \rightarrow L_p(\mu, X/Y)$ by $\Phi(f) = \pi \circ f$.
	
	$\Phi$ is a bounded linear map and $ker(\Phi) = L_p(\mu, Y)$, which follows from $\int \|\pi(f(w))\|^p d\mu(w) \leq \int \|f(w)\|^p d\mu(w)$.
	
	It remains to prove that for any $f\in L_p(\mu, X)$, $\|f+L_p(\mu, Y)\|=\|\Phi (f)\|$. It is clear that for any $g\in L_p(\mu, Y)$, $\|f+g\|_p\geq \|\Phi (f)\|_p$, we now prove the other inequality.
	
	Using Bartle-Graves theorem (see \cite[Pg. 184]{HR}) we have a continuous cross section map $ \rho: X/Y \rightarrow X$, such that $\rho(\pi(x)) \in \pi(x)$.
	For $g \in L_p(\mu, X/Y)$, we now see that $\rho \circ g \in L_p(\mu,X)$ and   $\Phi(\rho \circ g) = \pi \circ (\rho \circ g) = g$.
	Thus $\Phi$ is onto.
	
	For any $\e >0$, by taking $\la = 1+ \e$ (as stated in \cite[Pg. 184]{HR})  we see that $\Phi$ is a quotient map, or in other words $\|f+L_p(\mu, Y)\|=\|\Phi (f)\|$
	i.e., $L_p(\mu,X)/L_p(\mu,Y)$ is isometric to $L_p(\mu, X/Y)$.
\end{proof}

For the remaining part we may assume without loss of generality that $\mu$ is a purely non-atomic measure.

\begin{prop}\label{P4}
	Let $(\Omega, {\Sigma}, \mu)$ be a probability space. Let $X$ be a Banach space and $Y$ be a subspace of $X$. Suppose $p, q\in (1,\iy)$ with $\frac{1}{p}+\frac{1}{q}=1$ and if $p=1$ then $q=\infty$.
	\bla
	\item If $X^*$ has the RNP then $L_p(\mu, Y)^\bot = L_q(\mu, Y^\bot)$.
	\item If $L_p(\mu, Y)^\bot = L_q(\mu, Y^\bot)$ then $Y^\perp$ has the RNP.
	\el
\end{prop}

\begin{proof}
	$(a).~$ We always have $L_q(\mu, Y^\perp) \subseteq L_p(\mu, Y)^\perp$.
	
	To prove the other inclusion, we follow a similar arguments as stated in the proof of \cite[Theorem~1, Pg. 98]{DU}.
	
	RNP being hereditary property $Y^\bot$ has the RNP.
	
	Let $g \in (L_p(\mu, X))^ \ast = L_q(\mu,X^\ast)$. Suppose $g \in L_p(\mu, Y)^\bot$.
	Consider the vector measure $G: {\Sigma} \rightarrow X^\ast$ defined by $G(E)(x) = \int_E g(w)(x) d\mu(w)$.
	As in the proof of \cite[Theorem~1, Pg. 98]{DU}, one has that $G$ is a countably additive vector measure of bounded variation with respect to $\mu$.
	
	We now claim that $G$ is a $Y^\bot$-valued measure.
	
	Let $y \in Y$ and $E \in {\Sigma}$. Since $y\chi_E \in L_p(\mu, Y)$, by the choice of $g$ we have,
	$$\int_E g(w)(y) d\mu(w) = 0.$$
	Hence $G(E)(y) = 0$ for all $y$. Thus $G$ is a $Y^\bot$-valued measure. Since $Y^\bot$ has the RNP, $G$ has a $Y^\bot$-valued derivative $h$. But by the uniqueness of the derivative $g = h$ a.e. Thus $g \in L_q(\mu,Y^\perp)$.
	
	Identical arguments for $p=1$ and the fact $L_1 (\mu, X)^* = L_\iy (\mu, X^*)$ lead to the proof for $p=1$.
	
	$(b).~$ Now $L_q(\mu, Y^\perp)= L_q(\mu, (X/Y)^*)=L_p(\mu, Y)^\perp$.
	
	Also $\left(L_p(\mu, X)/L_p(\mu, Y)\right)^*=L_p(\mu, X/Y)^*$. Thus $L_p(\mu,X/Y)^*=L_q(\mu, Y^\perp)$. As $\mu$ is non atomic, by \cite[Theorem~1, Pg. 98]{DU}, we have that $(X/Y)^*=Y^\perp$ has the RNP.
\end{proof}

We need the following observation in the subsequent discussion.

\begin{rem}
	Let $\mu$ be probability measure then for $f\in X^*$,
	\begin{align*}
	\|\chi_\Omega f|_{L_p(\mu,Y)}\|&=\ds\sup_{g\in B_{L_p(\mu,Y)}}|\int_{\Omega}\langle g,\chi_{\Omega}f\rangle d\mu|\\
	&=\ds\sup_{g\in B_{L_p(\mu,Y)}}|\int_{\Omega}\langle g,\chi_{\Omega}f|_Y\rangle d\mu|\\
	&\leq \ds\sup_{g\in B_{L_p(\mu,Y)}}\int_{\Omega}\|f|_Y\|\|g(t)\|d\mu(t)\\
	&=\|f|_Y\|\ds\sup_{g\in B_{L_p(\mu,Y)}}\|g\|_1\leq \|f|_Y\|.
	\end{align*}
	Again let $\e>0$, there exist $y\in B_Y$ such that $f|_Y(y)\geq \|f|_Y\|-\e$ which implies $\int_{\Omega}\langle \chi_{\Omega}y,\chi_{\Omega}f\rangle d\mu\geq \int_{\Omega}\chi_{\Omega}(\|f|_Y\|-\e)d\mu=\|f|_Y\|-\e$. Therefore $\|\chi_\Omega f|_{L_p(\mu,Y)}\|\geq|\int_{\Omega}\langle\chi_{\Omega}y,\chi_{\Omega}f\rangle d\mu|\geq \int_{\Omega}\langle\chi_{\Omega}y,\chi_{\Omega}f\rangle d\mu\geq \int_{\Omega}\chi_{\Omega}(\|f|_Y\|-\e)d\mu=\|f|_Y\|-\e$ and $\e$ is arbitrary positive. So, we get $\|\chi_\Omega f|_{L_p(\mu,Y)}\|\geq\|f|_Y\|$. Hence $\|\chi_\Omega f|_{L_p(\mu,Y)}\|=\|f|_Y\|$.
\end{rem}

\begin{thm}\label{T8}
	Let $(\Omega, \Sigma, \mu)$ be a probability space. Let $Y$ be a subspace of a Banach space $X$ such that $X^*$ has RNP. Then,
	\bla
	\item $Y$ has property-$U$ in $X$ if and only if $L_p(\mu, Y)$ has property-$U$ in $L_p(\mu, X)$, for $1< p<\iy$.
	\item $Y$ has property-$SU$ in $X$ if and only if $L_p(\mu, Y)$ has property-$SU$ in $L_p(\mu, X)$, for $1< p<\iy$.
	\el
\end{thm}

\begin{proof}
	$(a).$ Let us assume that $Y$ has property-$U$ in $X$. It remains to show that $L_p(\mu, Y)^\perp$ is Chebyshev in $L_p(\mu, X)^*$. Being a $w^*$-closed subspace of a dual space, $L_p(\mu, Y)^\perp$ is proximinal. Now for any $\varphi\in L_p(\mu, X)^*\cong L_q(\mu, X^*)$, $d(\varphi, L_q(\mu, Y^\perp))=\|d(\varphi (.), Y^\perp)\|_q$, see \cite[Lemma~2.10]{LC}. Hence the uniqueness of the best approximation from $\varphi$ to $L_q(\mu, Y^\perp)$ follows from the uniqueness of best approximation from $\varphi (\omega)$ to $Y^\perp$ for all $\omega\in \Omega$ a.e. $[\mu]$.
	
	For the converse we follow the equivalence of $(1)$ and $(3)$ in \cite[Theorem~2.1]{L}. Assume that $f_1,f_2\in Y^{\#}$ and $f_1+f_2\in Y^\perp$. Since  $(f_1+f_2)(y)=0$ for all $y\in Y$ we have $\chi_{\Omega}f_1+\chi_{\Omega}f_2\in L_p(\mu,Y)^\perp$. It is clear that $\chi_{\Omega}f_i\in L_q(\mu,X^*)$ and $\|\chi_{\Omega}f_i|_{L_p(\mu,Y)}\|=\|\chi_{\Omega}f_i\|=\|f_i\|$ for $i=1,2$. Therefore $\chi_{\Omega}f_1,\chi_{\Omega}f_2\in L_p(\mu,Y)^{\#}$. As $L_p(\mu,Y)$ has property-$U$ in $L_p(\mu,X)$ we have $\chi_{\Omega}f_1+\chi_{\Omega}f_2=0$, hence $f_1+f_2=0$.
	
	$(b).$ Let us recall the characterizations for property-$SU$ in \cite{O} (as stated in Section~ 1).
	Let us assume that $Y$ has property-$SU$ in $X$. It remains to prove that $L_q(\mu, Y^\perp)$ is an ideal in $L_q(\mu, X^*)$. This follows from Proposition~\ref{P4} and the easy observation $\tilde{P}:L_q(\mu, X^*)\ra L_q(\mu, X^*)$ defined by $\tilde{P}(g)=P\circ g$, where $P:X^*\ra X^*$ is a contractive projection with $ker (P)=Y^\perp$.
	
	For the converse we follow the equivalence of $(1)$ and $(6)$ in \cite[Pg. 1]{O}. Assume that $f_1,f_2,f_3\in Y^{\#}$ and $f_1+f_2+f_3\in Y^\perp$. If $L_p(\mu,Y)$ has property-$SU$ in $L_p(\mu,X)$ using similar arguments as stated above we can show $f_1+f_2+f_3=0$, hence $Y$ has property-$SU$ in $X$.
\end{proof}

\begin{rem}
	\bla
	\item Note that the converse of both the statements in Theorem~\ref{T8} does not require the condition that the space $X^*$  satisfy RNP.
	\item We also note that the conclusion in Theorem~\ref{T8} may not hold true for $p=1$.
	\el
\end{rem}

We now come to  the spaces of type $C(K,X)$, where $X$ is a Banach space and $K$ is a compact Hausdorff space. Let us recall $C(K,X)\cong C(K)\otimes_{\e}^{\vee} X$, the injective tensor product of $C(K)$ and $X$, as stated in Section~2.

\begin{thm}\label{T7}
	Let $X$ be an $L_1$-predual space and $Z$ be a Banach spaces such that $Z^*$ has RNP. Let $Y$ be a subspace of $Z$. Then, \bla
	\item $Y$ have property-$U$ in $Z$ if and only if $X\otimes_\e^\vee Y$ has property-$U$ in $X\otimes_\e^\vee Z$.
	\item $Y$ have property-$SU$ in $Z$ if and only if $X\otimes_\e^\vee Y$ has property-$SU$ in $X\otimes_\e^\vee Z$.
	\el
\end{thm}

\begin{proof}
	$(a).$ Let $Y$ have property-$U$ in $Z$.
	
	We now show that, $\left(X \otimes_{\e}^{\vee} Y\right)^\perp$ is Chebyshev in $(X\otimes_{\e}^{\vee}Z)^*$.
	
	Since $Z^*$ has the RNP and $X^*$ has approximation property (see \cite[Pg. 73]{RR1}), $(X\otimes_{\e}^{\vee}Z)^*\cong X^*\otimes_{\pi}^{\wedge} Z^*$. On the other hand $(X \otimes_{\e}^{\vee} Y)^\perp\cong \left(X \otimes_{\e}^{\vee} Z/X \otimes_{\e}^{\vee} Y\right)^*$. Since $X$ is an $L_1$-predual it follows from \cite[Corollary~18]{MR} that $\left(X \otimes_{\e}^{\vee} Z/X \otimes_{\e}^{\vee} Y\right)\cong X\otimes_{\e}^{\vee} (Z/Y)$ and hence $\left(X\otimes_{\e}^{\vee} Y\right)^\perp\cong X^*\otimes_{\pi}^{\wedge} Y^\perp\cong L_1(\mu)\otimes_{\pi}^{\wedge} Y^\perp \cong L_1(\mu, Y^\perp)$, the last identity follows from the fact that $X^*\cong L_1(\mu)$ for some positive measure space $(\Omega,\Sigma,\mu)$ and the properties of projective tensor product. Now $\left(X\otimes_{\e}^{\vee} Z\right)^*\cong L_1(\mu, Z^*)$. Being a $w^*$-closed subspace of $L_1(\mu, Z^*)$, $L_1(\mu, Y^\perp)$ is proximinal.
	It remains to prove that best approximations in $L_1(\mu,Y^\bot)$ are unique. This follows from similar arguments used in the proof of Theorem~\ref{T8}$(a)$.
	
	Conversely, let $Y$ does not have property-$U$ in $Z$. Let $f\in Y^*$ and $g_1$, $g_2$ be two norm preserving extensions in $Z^*$. Let $w\in S_{X^*}$, then $w\otimes g_1$, $w\otimes g_2\in X^*\otimes_\pi^\wedge Z^*(=(X\otimes_\e^\vee Z)^*)$ are two norm preserving extensions of $w\otimes f\in X^*\otimes_\pi^\wedge Y^*(=(X\otimes_\e^\vee Y)^*)$, which contradicts our assumption.
	
	$(b).$ As the ideal property is stable under injective tensor product (see \cite[Lemma 2, pg-601]{RA}), the result is true for property-$SU$.
	
	Conversely assume that $X\otimes_\e^\vee Y$ has property-$SU$ in $X\otimes_\e^\vee Z$. Again we follow the equivalence for property-$SU$ in \cite{O}.
	
	Let $f_1,f_2,f_3\in Y^\#$ such that $f_1+f_2+f_3\in Y^\perp$. It remains to show that $f_1+f_2+f_3=0$. Let $g\in S_{X^*}$ then $g\otimes f_i\in (X\otimes_\e^\vee Y)^\#$, $i=1,2,3$. In fact, for fixed $i$, $\|g\otimes f_i\|=\|g\|\|f_i\|=\|f_i\|$. Let $\e>0$, there exist $x_0\in S_X$ such that $|g(x_0)|>1-\e$. It is now clear that,
	$\|g\otimes f_i|_{X\otimes_\e^\vee Y}\|\geq (1-\e)\|f_i\|$. Hence $\|g\otimes f_i|_{X\otimes_\e^\vee Y}\|=\|g\otimes f_i\|$, $i=1,2,3$.
	
	It is now clear that, $g\otimes f_1+g\otimes f_2+g\otimes f_3\in (X\otimes_\e^\vee Y)^\perp$. In fact for simple tensor $x\otimes y$, $(g\otimes f_1+g\otimes f_2+g\otimes f_3) (x\otimes y)=0$. Hence, if $D$ be the linear span of all simple tensors in $X\otimes_\e^\vee Y$ then $(g\otimes f_1+g\otimes f_2+g\otimes f_3)|_D=0$. From the density of $D$ we have $f_1+f_2+f_3=0$.
\end{proof}

We do not know whether a similar conclusion derived in Theorem~\ref{T7} also holds for the projective tensor product of the spaces.

As an application of Theorem~\ref{T7} we have the following for the subspaces of type $C(K,Y)$ in $C(K,X)$.

\begin{cor} \label{H1}
	Let $X$ be a Banach space such that $X^*$ has RNP and $Y$ be a subspace of $X$. Then $Y$ has property-$U$ (property-$SU$) in $X$ if and only if $C(K,Y)$ has property-$U$ (property-$SU$) in $C(K,X)$.
\end{cor}

\section{Property-$(U)$ and $(SU)$ in Quotient spaces}

The aim of this section is to discuss property-$U$ and $SU$ in the quotient spaces under suitable assumptions on the underlying spaces. The first observation ensures that if a subspace of a Banach space satisfies one of these properties then it also transfers to the quotient spaces. Converse to this result is derived in Theorem~\ref{TP3} with suitable assumptions on the respective spaces.

\begin{thm}\label{H4}
	\bla
	\item Let $Z \ci Y \ci X$ be closed subspaces of $X$, where $Y$ has property-$SU$ in $X$, then $Y/Z$ has property-$SU$ in $X/Z$.
	\item Let $X$ be an $L_1$-predual. Let $Y$ be a subspace of $X$ which has property-$SU$ and $J$ be an M-ideal in $Y$. Then $Y/J$ is an $L_1$-predual.
	\el
\end{thm}

\begin{proof}
	$(a).$ We first show that $Y/Z$ has property-$U$ in $X/Z$.
	
	Consider an element $f\in Z_{Y^*}^\perp$, with two norm preserving extensions, $g, h
	\in Z_{X^*}^\perp $.
	
	Now $g|_Y= h|_Y$, since $Y$ has property-$U$ we have $g=h$.
	
	For the remaining part, let $P: X^{**} \ra Y^{\perp\perp}$ be a contractive projection.
	Define $P^\prime: X^{**}/Z^{\bot\bot} \ra Y^{\bot\bot}/Z^{\bot\bot}$ by $P^\prime(\pi(\tau))= \pi(P(\tau))$.
	For $\tau \in X^{**}$, where $\pi$ is the quotient map on either quotient space.
	It is easy to see, $P^\prime$ is a contractive projection.
	
	$(b).$ Suppose $X^*=L_1(\mu)$, for some measure $\mu$ and $P_1:X^*\to Y^\#$ be a norm-$1$ projection for some subspace $Y^\#$ of $X^*$ such that $ker(P_1)=Y^\perp$. Being an image of a contraction of $L_1(\mu)$, $Y^\#$ is isometric with $L_1(\nu)$ for some measure $\nu$ (\cite[Theorem~3, Chapter~6]{HE}). It is clear that $Y^\#$ is isomorphic with $Y^*$ and let $P_2:Y^*\to J_{Y^*}^\perp$ be an $L$-projection. Hence it follows that $J_{Y^*}^\perp$ is isometric with $L_1(\lambda)$. Now $J_{Y^*}^\perp$ is isometrically isomorphic to $(Y/J)^*$, in other words $Y/J$ is $L_1$-predual space.
\end{proof}

\begin{prop}
	Let $X$ be an $L_1$-predual. Let $Y$ be a subspace of $X$ which has property-$SU$ and $J$ be an M-ideal in $Y$. Then $J_{Y^*}^\perp$ is isometrically isomorphic with $J_{X^*}^\perp/{Y^{\perp}}$ and  $J_{X^*}^\perp/{Y^{\perp}}$ is isometric with $L_1(\upsilon)$ for some measure $\upsilon$.
\end{prop}
\begin{proof}
	Define $\phi:J_{Y^*}^\perp\to J_{X^*}^\perp/{Y^{\perp}}$ by the following. $\phi(f)=\tilde{f}+Y^{\perp}$, where $\tilde{f}$ be the unique Hahn-Banach extension of $f$. Clearly $\|\phi(f)\|=\|\tilde{f}+Y^{\perp}\|=\|\tilde{f}_{|Y}\|=\|f\|$ and $\|\phi^{-1}(g+Y^{\perp})\|=\|g_{|Y}\|=\|g+Y^{\perp}\|$. So $J_{Y^*}^\perp$ isometrically isomorphic to $J_{X^*}^\perp/{Y^{\perp}}$. And from above $J_{X^*}^\perp/{Y^{\perp}}$ is isometric with $L_1(\upsilon)$ for some measure $\upsilon$.
\end{proof}

\begin{prop}\label{P3}
	Let $Z \ci Y \ci X$. Suppose $Y$ has property $n. X.I.P.$ in $X$
	and $Z$ is a proximal subspace of $X$. Then $Y/Z$ has $n. X/Z. I. P.$ in $X/Z$.
\end{prop}

\begin{proof}
	Let $\pi$ denote the quotient map. Let $\{B(\pi(y_i),r_i)\}_{i=1}^n$ be $n$-closed balls with centers in $Y/Z$ and let $x \in X$ be such that $\|x+Z-\pi (y_i)\| = d(x-y_i,Z) \leq r_i$ for $1\leq i \leq n$. Since $Z$ is proximal in $X$, let $z_i \in Z$ be such that $\|x-y_i-z_i\| = d(x-y_i,Z) \leq r_i$ for $1 \leq i \leq n$.
	Now for the collection $\{B(y_i-z_i,r_i)\}_{i=1}^n$ of balls with centers in $Y$, as $Y$ has $n.X.I.P.$ in $X$, for $\e >0$, there is a $y_0 \in Y$
	such that $\|y_0-y_i-z_i\| \leq r_i + \e$ for $1 \leq i \leq n$.
	
	Now $\pi(y_0) \in Y/Z$ and $\|\pi(y_0)-\pi(y_i)\|= \|\pi(y_0-y_i-z_i)\| \leq \|y_0-y_i-z_i\| \leq r_i+ \e$ for $1 \leq i \leq n$. Hence $Y/Z$ has the $n. X/Z. I. P$ in $X/Z$.
\end{proof}

Let $Y$, $Z$ be subspaces of $X$ such that $Z\ci Y\ci X$, the converse of Theorem~\ref{H4} is true when we are assuming $Z$ is an $M$-ideal in $X$. Before going to proof of it we first observe the following.

\begin{prop}
	Let $Y$, $Z$ be subspaces of $X$ such that $Z\ci Y\ci X$. Then $(Y/Z)^\#=\{g\in Z^\perp_{X^*}:\|g|_Y\|=\|g\|\}$.
\end{prop}

\begin{proof}
	Let $g\in Z^\perp_{X^*}$ and $\Lambda\in (X/Z)^*(=Z^\perp_{X^*})$ is the corresponding element of $g$. Now $\|\Lambda|_{Y/Z}\|=\ds\sup_{\|y+Z\|\leq 1}|\Lambda(y+Z)|=\ds\sup_{\|y+Z\|\leq 1}|g(y)|\geq \ds\sup_{\|y\|\leq 1}|g(y)|=\|g|_Y\|$.
	Again $|\Lambda|_{Y/Z}(y+Z)|=|\Lambda(y+Z)|=|g(y)|=|g|_Y(y)|=|g|_Y(y+z)|\leq \|g|_Y\|\|y+z\|$ and this is true for all $z\in Z$, hence $|\Lambda|_{Y/Z}(y+Z)|\leq \|g|_Y\|\|y+Z\|$. Therefore $\|\Lambda|_{Y/Z}\|= \|g|_Y\|$. As $(Y/Z)^\#=\{\Lambda\in (X/Z)^*:\|\Lambda|_{Y/Z}\|=\|\Lambda\|\}$ and $\|\Lambda\|=\|g\|$ we have $(Y/Z)^\#=\{g\in Z^\perp_{X^*}:\|g|_Y\|=\|g\|\}$.
\end{proof}

We now come to the main result of this section.

\begin{thm} \label{TP3}
	Let $Y, Z$ be subspaces of $X$ such that $Z\ci Y\ci X$. Suppose $Z$ is an M-ideal in $X$ and $Y/Z$ has property-$SU$ in $X/Z$. Then $Y$ has property-$SU$ in $X$.
\end{thm}

\begin{proof}
	Let us consider the following cases to conclude the desired result. We first show that $Y$ has property-$U$ in $X$.
	
	Choose $f_1, f_2\in Y^\#$ with $f_1+f_2\in Y^\perp$.
	
	{\sc Case~1:} If $f_1, f_2\in Z^\#$ then as $f_1+f_2\in Z^\perp$ we have $f_1+f_2=0$.
	
	{\sc Case~2:} Let $f_1=p_1+q_1$ and $f_2=p_2+q_2$ where $p_i\in Z^\#$ and $q_i\in Z^\perp$.
	
	Now given that $X^*=Z^\# \bigoplus_{\ell_1} Z^\perp$. Hence $\|f_1+f_2\|=\|p_1+p_2\|+\|q_1+q_2\|$.
	
	{\sc Claim~2.1:} $p_1+p_2=0$.
	
	In fact for any $z\in Z$, $q_i(z)=0$ and $(f_1+f_2)(z)=0$, we have $(p_1+p_2)(z)=0$. This follows that $p_1+p_2\in Z^\perp$. Now from {\sc Case~1} it follows that $p_1+p_2=0$.
	
	{\sc Claim~2.2:} $q_1+q_2=0$.
	
	Now $0=\|(f_1+f_2)|_Y\|=\|(p_1+p_2)|_Y\|+\|(q_1+q_2)|_Y\|=\|(q_1+q_2)|_Y\|$, the second identity follows from the fact that $Z$ is an M-ideal in $Y$.
	
	Hence we have $q_1+q_2\in Y^\perp$. We now show that $q_i\in (Y/Z)^\#$.
	
	Again since $Z$ is an M-ideal in $Y$ and $\|f_i\|=\|f_i|_Y\|$, we have $\|p_i\|+\|q_i\|=\|p_i|_Y\|+\|q_i|_Y\|$, $i=1, 2$. Now as $\|p_i\|=\|p_i|_Z\|\leq \|p_i|_Y\|\leq \|p_i\|$ we have $\|q_i\|=\|q_i|_Y\|$, which in other words $q_i\in Y_{Z^\perp}^\#$.
	
	Since $Y/Z$ has property-$SU$ in $X/Z$, we have $q_1+q_2=0$ and hence $f_1+f_2=0$.
	
	{\sc Case~3:} If one of $f_1, f_2$ is in $Z^\#$.
	
	{\sc Case~3.1:} Let $f_1\in Z^\#$ and $f_2\in Z^\perp$. Then $0=\|(f_1+f_2)|_Y\|=\|f_1|_Y\|+\|f_2|_Y\|=\|f_1\|+\|f_2\|=\|f_1+f_2\|\Ra f_1+f_2=0$.
	
	{\sc Case~3.2:} Let $f_1\in Z^\#$ and $f_2=p_2+q_2$ where $p_2\in Z^\#$ and $q_2\in Z^\perp$.
	Now given that $X^*=Z^\# \bigoplus_{\ell_1} Z^\perp$. Hence $\|f_1+f_2\|=\|f_1+p_2\|+\|q_2\|$.
	
	{\sc Claim~3.2.1:} $f_1+p_2=0$.
	
	In fact for any $z\in Z$, $q_2(z)=0$ and $(f_1+f_2)(z)=0$, we have $(f_1+p_2)(z)=0$. This follows that $f_1+p_2\in Z^\perp$. Now from {\sc Case~1} it follows that $f_1+p_2=0$.
	
	{\sc Claim~3.2.2:} $q_2=0$.
	
	Now $0=\|(f_1+f_2)|_Y\|=\|(f_1+p_2)|_Y\|+\|q_2|_Y\|=\|q_2|_Y\|$, the second identity follows from the fact that $Z$ is an M-ideal in $Y$.
	
	Hence we have $q_2\in Y^\perp$. We now show that $q_2\in (Y/Z)^\#$.
	
	Again since $Z$ is an M-ideal in $Y$ and $\|f_2\|=\|f_2|_Y\|$, we have $\|p_2\|+\|q_2\|=\|p_2|_Y\|+\|q_2|_Y\|$ and $\|p_1\|=\|p_1|_Y\|$. Now as $\|p_2\|=\|p_2|_Z\|\leq \|p_2|_Y\|\leq \|p_2\|$ we have $\|q_2\|=\|q_2|_Y\|$, which in other words $q_2\in Y_{Z^\perp}^\#$.
	
	Since $Y/Z$ has property-$SU$ in $X/Z$, we have $q_2=0$ and hence $f_1+f_2=f_1+p_2+q_2=0$.
	
	It remains to prove that $Y$ is an ideal in $X$. Now by Theorem~\ref{T1}, as $Y$ has property-$U$, it is enough to check that $Y$ has $3.X.I.P.$ in $X$.
	
	Let $\{B(y_i, r)\}_{i=1}^3$ be $3$ closed balls in $X$ with centres in $Y$. Suppose $\bigcap_i B(y_i, r)\neq\es$. Choose $x\in \bigcap_i B(y_i, r)$. Then it is clear that $x+Z\in B(y_i+Z, r)$, where the balls are now taken in the quotient space $X/Z$.
	
	Now as $Y/Z$ has property-$SU$ we have $\left(\bigcap_{i=1}^3 B(y_i+Z, r+\e)\right)\cap (Y/Z)\neq\es$. Let us choose $y_0+Z\in \left(\bigcap_{i=1}^3 B(y_i+Z, r+\e)\right)\cap (Y/Z)$. In other words $\|y_i-y_0+Z\|\leq r+\e$ for $1\leq i\leq 3$. Get $z_i\in Z$ such that $\|y_i-y_0-z_i\|<r+2\e$.
	
	We now use a characterization of M-ideal stated in \cite[Pg. 18]{HW}.
	Consider the $3$ balls $\{B(y_i-y_0, r+2\e)\}_{i=1}^3$ in $X$. Each ball intersects $Z$, as $z_i\in B(y_i-y_0, r+2\e)$, and finally $x-y_0\in \bigcap_{i=1}^3 B(y_i-y_0, r+2\e)$. Because $Z$ is an M-ideal in $X$, there must exists a $z_0\in \left(\bigcap_{i=1}^3 B(y_i-y_0, r+3\e)\right)\cap Z$. Which concludes that $y_0+z_0\in \bigcap_{i=1}^3 B(y_i, r+3\e)$.
	
	As $\e$ is arbitrary the result follows.
\end{proof}

\begin{rem}
	From the proof of Theorem~\ref{TP3}, it is clear that if $Z$ is a semi M-ideal in $X$ and $Y/Z$ has property-$U$ in $X/Z$ then $Y$ has property-$U$ in $X$. In fact, if $Z$ is a semi M-ideal in $X$ then $\|x^*\|=\|Px^*\|+\|x^*-Px^*\|$, where $x^*\in X^*$ and $P:X^*\ra Y^\perp$ is a projection. This decomposition leads to that, $Z$ has property-$U$ in $X$ and hence \cite[Theorem~2.1, Pg. 99]{L} applies for $Z$. Hence all the cases in Theorem~\ref{TP3} devoted to prove $Y$ has property-$U$ in $X$ can be fitted for this case also.
\end{rem}

\section{Property-$U$ and $SU$ for spaces of type $L_1$-predual}

In this section we discuss property-$U$ and $SU$ for ideals and other subspaces of $L_1$-preduals.
An easy argument concerning property-$U$ and the $3.X.I.P.$ ensure the following.

\begin{thm}\label{T6}
	Let $X$ be a Banach space. $Y$, $Z$ are subspaces of $X$, and $\overline{Z}=X$. If $Y$ has property-$U$ ($SU$) in $\text{span}\{Y\cup \{z\}\}$, for all $z\in Z$ then $Y$ has property-$U$ ($SU$) in $X$.
\end{thm}

Using characterization of $L_1$-preduals discussed in \cite[Chapter~7]{HE} we now have a similar result for ideals in Theorem~\ref{T6}.

\begin{thm}\label{T30}
	Let $X$ be an $L_1$-predual space. Suppose $Z \ci X$ is a dense subspace and let $Y \ci X$ be a closed subspace such that for any $z \in Z\sm Y$, $Y \ci span\{z,Y\}$ is an ideal if and only if $Y$ is an ideal in $X$.
\end{thm}

Let us recall that a subspace $W$ of a Banach space $X$ is said to be 1-complemented if there exists a linear onto projection $P:X\ra W$ with $\|P\|=1$.
It is clear that if $X$ is a Banach space then $X^{\perp\perp}(\cong X^{**})$ is 1-complemented in $X^{****}$. This leads to conclude that, $Y$ is an ideal in $X$ if and only if $Y^{\perp\perp}$ is an ideal in $X^{**}$.

Now observe that for a Banach space $X$ of type $L_1$-predual or M-embedded, $X^*$ is an $L$-summand in $X^{***}$. Hence we have the following.

\begin{thm}\label{T10}
	Let $X$ be an $L_1$-predual (or an $M$-embedded) space, a finite co-dimensional subspace $Y$ of $X$ has property-$U$ if and only if $Y^{\perp\perp}$ has property-$U$ in $X^{**}$.
\end{thm}
\begin{proof}
	The result follows from the above discussion and the fact that for any finite linear functionals $(f_i)_{i=1}^n$ in a dual space $X^*$, $\left(\bigcap_i ker (f_i)\right)^{\perp\perp}=\bigcap_i ker (\tilde{f_i})$, where $\tilde{f_i}$'s are the canonical images of $f_i$'s in $X^{***}$. Hence $Y^\perp=sp\{f_i:1\leq i\leq n\}$, as a Chebyshev subspace of $X^*$, continues to be Chebyshev in $X^{***}$.
\end{proof}

Hence we have the following.

\begin{thm}\label{T2}
	Let $X$ be an $L_1$-predual (or an $M$-embedded) space and $Y$ be a finite co-dimensional subspace of $X$. Then $Y$ has property-$SU$ in $X$ if and only if $Y^{\perp\perp}$ has property-$SU$ in $X^{**}$.
\end{thm}

We now show that the condition of finite codimensionality can not be omitted in Theorem~\ref{T2}.

\begin{ex}\label{Ex1}
	There exists a one-dimensional subspace of $C[0,1]$ which has property-$U$ but does not have property-$U$ in $C[0,1]^{**}$.
	Let \[
	f(t) =
	\begin{cases}
	2t & \text{if $0\leq t \leq \frac{1}{2}$} \\
	2(1-t) & \text{if $\frac{1}{2}\leq t \leq 1$},
	\end{cases}
	\] $f$ is the smooth point of $C[0,1]$ and $\delta_{\frac{1}{2}} \in M[0,1]$ is the unique linear functional which attains its norm at $f\in C[0,1]$. Hence $Y=\text{span}\{f\}$ has property-$U$ in $C[0,1]$.
	
	Now suppose $Y$ has property-$U$ in $C[0,1]^{**}$, from the assumption, it is easy to see that $\delta_{\frac{1}{2}}$ is a point of continuity of identity map $I:(B_{X^*},w^*)\to (B_{X^*},w)$, where $X=C[0,1]$. Let $t_n \to \frac{1}{2}$, $t_n \neq \frac{1}{2}$, for all $n\in \mathbb{N}$, hence $f(t_n)\to f(\frac{1}{2})$, for all $f\in C[0,1]$, i.e $\delta_{t_n} \to \delta_{\frac{1}{2}}$ in weak* topology. Since identity map is weak*-weak continuous at the point $\delta_{\frac{1}{2}}$, $\delta_{t_n} \to \delta_{\frac{1}{2}}$ in weak topology, i.e $\delta_{\frac{1}{2}}\in \overline{\text{conv}}^w\{\delta_{t_n}:n\in \mathbb{N}\}$, hence $\delta_{\frac{1}{2}}\in \overline{\text{conv}}^{\|.\|}\{\delta_{t_n}:n\in \mathbb{N}\}$. Hence there exists a sequence $(\mu_n) \ci \text{conv}\{\delta_{t_n}:n\in \mathbb{N}\}$ such that $\mu_n \to \delta_{\frac{1}{2}}$. Thus there exists $n_0 \in \mathbb{N}$ such that $\|\mu_{n_0}-\delta_{\frac{1}{2}}\| < \e$, hence $\|\lambda \delta_{t_{n_1}}+(1-\lambda)\delta_{t_{n_2}}-\delta_{\frac{1}{2}}\| < \e$, for some $\la \in [0,1]$ and $n_1, n_2 \in \mathbb{N}$, but $\|\lambda \delta_{t_{n_1}}+(1-\lambda)\delta_{t_{n_2}}-\delta_{\frac{1}{2}}\|=2$, hence contradiction.
\end{ex}

Let $X$ be a Banach space. It is known that for any two $M$-ideals, $Y,Z \ci X$, the sum space $Y+Z$ is a closed $M$-ideal in $X$. See \cite[Proposition~I.1.11.]{HW}. We next show that the validity of the same question for property-$SU$, determines a Hilbert space in the class of Banach spaces which are smooth. Note that a Hilbert space of dimension bigger than $1$, has no non-trivial $M$-ideals, where as,  in any Hilbert space, closed subspaces have property-$SU$

\begin{thm}
	Let $X$ be a smooth, Banach space of dimension $>3$. Suppose for every $Y,Z \ci X$ subspaces having property-$SU$ in $X$ and with sum $Y+Z$ is closed, $Y+Z$ also has property-$SU$.
	Then $X$ is isometric to a Hilbert space.
\end{thm}

\begin{proof}
	We note that since $X$ is a smooth space, any one dimensional subspace of $X$ has property-$SU$ in $X$. Thus by our hypothesis any two dimensional space has property-$SU$,
	so it has property-$U$ and is the range of a projection of norm-1. Now by  induction, every finite dimensional subspace has the property-$SU$.
	Thus  any finite dimensional subspace is the range of a projection of norm one. Now by Kakutani's theorem (See \cite[Pg.150]{HE}) we get that $X$ is a Hilbert space.
\end{proof}

\section{A few examples}

In this section several examples are given satisfying properties $U$ and $SU$. Our first example characterizes finite co-dimensional subspaces of $c_0$ with property-$U$.

Let us recall that, $f\in c_0^*$ where $ker (f)$ is proximinal if and only if $f$ has only finite support and if $Y\ci X$ is a proximinal subspace of finite co-dimensional and $Y\ci Z\ci X$ then $Z$ is proximinal in $X$(see \cite{IS}).

\begin{prop}
	Let $Y \ci c_0$ be a proximinal subspace of finite codimension. $Y$ has  property-$U$ in $c_0$ if and only if $c_0/Y$ is isometric to a subspace, with property-$U$, of $\ell_\iy(k)$ for some integer $k\leq dim (c_0/Y)$.
\end{prop}

\begin{proof}
	Since $Y$ is proximinal and finite codimensional, one has,
	$Y^\bot = span\{f_1,...f_m\}$ for some $f_i \in \ell_1$, each having only finitely many non-zero coordinates. Thus we may assume for some $k>0$, $f_i(j)=0$ for $j>k$ and $1\leq i \leq m$. It is easy to see now, $Y = F \bigoplus_{\iy}\{y \in c_0:y(j)=0 \mbox{~for~all}~j\leq k\}$, for some finite dimensional subspace of $c_0$. By hypothesis $F$ is a subspace with property-$U$ of $\ell_\iy(k)$. Since $c_0 =\ell_\iy(k) \bigoplus_{\iy} \{x \in c_0: x(j)=0\mbox{~for~all~}j\leq k\}$. $Y$ is a subspace with property-$U$ of $c_0$.
\end{proof}

In the next result, we show the dimension $n$ plays an important role for having property-$U$ of the subspace  $Y=\ker\{(1,1,\ldots, 1)\}\ci (\mathbb{R}^n, \|.\|_\iy)$.

\begin{prop} \label{P1}
	Let $Y=\{(y_1,\, y_2,\ldots,y_n):y_1+y_2+\dots+y_n=0\}$ be a subspace of $(\mathbb{R}^n,\|.\|_\iy)$. Then
	\bla
	\item $Y$ has property-$U$ in $(\mathbb{R}^n,\|.\|_\iy),$ when $n$ is odd.
	\item $Y$ does not have property-$U$ in $(\mathbb{R}^n,\|.\|_\iy),$ when $n$ is even.
	\el
\end{prop}

\begin{proof}
	{\sc Case~1:} When $n$ is odd. Let us assume $n=2k+1$.
	
	It remains to prove that $span\{{\sc 1}\}$ is a Chebyshev subspace of $(\mb{R}^n,\|.\|_1)$. Then for any $(x_1,x_2,\ldots,x_n)\in \mb{R}^n$, define $\al_1,\al_2,\ldots,\al_n$ such that $\{x_i:1\leq i\leq n\}=\{\al_i:1\leq i\leq n\}$ and $\al_1\leq\al_2\leq\ldots\leq\al_n$. Then $(\al_0,\al_0,\ldots,\al_0)\in \mb{R}^n$ is the unique best approximation of $(x_1,x_2,\ldots,x_n)$, where $\al_0=\al_{k+1}$.
	
	{\sc Case~2:} When $n$ is even. Let us assume $n=2k$.
	
	Suppose that $(x_1,x_2,\ldots,x_n)$ and $(\al_1,\al_2,\ldots,\al_n)$ stand with the similar meaning as {\sc Case~1}. Then for any $\be\in [\al_k,\al_{k+1}]$, $(\be,\be,\ldots,\be)\in\mb{R}^n$ is a best approximation from $(x_1,x_2,\ldots,x_n)$ to $span\{1\}$.
\end{proof}

An obvious question is to consider property-$SU$ for the subspace $Y$ in $(\mb{R}^3, \|.\|_\iy)$. It was shown in \cite{O} that when $n=3$, $Y$ fails to have property-$SU$ in $(\mb{R}^3, \|.\|_\iy)$, where $Y=\{(x,y,z)\in \mathbb{R}^3:x+y+z=0\}\ci (\mathbb{R}^3,\|.\|_\iy)$. We derive a similar conclusion in $(\mathbb{R}^3,\|.\|_1)$.

\begin{prop} Let $Y=\{(x,y,z)\in \mathbb{R}^3:x+y+z=0\}\ci (\mathbb{R}^3,\|.\|_1)$, then $Y$ does not have property-$SU$ in $(\mathbb{R}^3,\|.\|_1)$.
\end{prop}

\begin{proof}
	It is clear that $Y^\perp= span\{(1,1,1)\}$ is a Chebyshev subspace of $(\mathbb{R}^3,\|.\|_\infty)$. Indeed, for $(x,y,z)\in \mathbb{R}^3$, $d((x,y,z),Y^\perp)=\frac{x\vee y\vee z -x\wedge y\wedge z}{2}$ and unique nearest point of $(x,y,z)$ from $Y^\perp$ is $\frac{x\vee y\vee z +x\wedge y\wedge z}{2}(1,1,1)$.
	
	It remains to prove that $Y^\#$ is not a linear subspace. one can check $ext(B_Y)=\{\pm(\frac{1}{2},-\frac{1}{2},0),\pm(\frac{1}{2},0,-\frac{1}{2}),\pm(0,\frac{1}{2},-\frac{1}{2})\}$.
	
	Let $f=(l,m,n)\in Y^{\#}$. Then
	$$\|f\|=|l|\vee|m|\vee|n|=\ds\max_{(x,y,z)\in ext(B_Y)}|lx+my+nz|=\|f|_Y\|.$$
	
	We consider the following to evaluate $Y^{\#}$.
	
	It is known that $|l|\vee |m|=|\frac{l+m}{2}|+|\frac{l-m}{2}|$, for any two real scalars $l, m$.
	
	Let the maximum is attained at $(x,y,z)=\pm(\frac{1}{2},-\frac{1}{2},0)$, then $|l|\vee|m|\vee|n|=|\frac{l-m}{2}|=|l|\vee|m|-|\frac{l+m}{2}|$.
	
	{\sc Case 1:~} If $|l|\vee|m|\leq |n|$ then $ |n|=|l|\vee|m|-|\frac{l+m}{2}|\Ra l+m=0.$
	
	{\sc Case 2:~} If $|l|\vee|m|\geq |n|$ then $|l|\vee|m|=|l|\vee|m|-|\frac{l+m}{2}|\Ra l+m=0$.
	
	Similarly, if the maximum is attained at $(\pm(\frac{1}{2},0,-\frac{1}{2})$ and $\pm(0,\frac{1}{2},-\frac{1}{2})$, we get $l+n=0$ and $m+n=0$ respectively.

	Hence $Y^{\#}=\{(a,b,z),(a,z,b),(z,a,b):a+b=0, z\in \mathbb{R}\}$. Choose $u=(-1,1,0)\in Y^{\#}\,\&\, v=(0,1,-1)\in Y^{\#}$ but $u+v=(-1,2,-1)\notin Y^{\#}$. Therefore $Y$ does not have property-$SU$ in $(\mathbb{R}^3,\|.\|_\iy)$.
\end{proof}

Our next Theorem identifies hyperplanes $Y$ in $c_0$ with property-$SU$.
\begin{thm}\label{T4}
	Let $(a_n)\in S_{\ell_1}$ and $\sup_{n\in \mathbb{N}}|a_n|>\frac{1}{2}$. Then $\ker\{(a_n)\}\ci c_0$ has property-$SU$.
\end{thm}
\begin{proof}
	Let us first observe that $ker (a_n)$ is 1-complemented in $c_0$ (see \cite[Theorem~6.1]{BP}) hence is an ideal in $c_0$.
	
	Let $X=c_0$, $Y=\ker \{(a_n)\}$. Since $\ds\sup_{n\in \mathbb{N}}|a_n|>\frac{1}{2}$, there exist $N\in \mathbb{R}$ such that $|a_N|>\frac{1}{2}$. Note that, the $N$ is unique, as $(a_n)\in S_{\ell_1}$. Let $G=\ker\{e_N\}$. It is clear that $G$ is the complement of $Y^{\perp}=\text{span}\{(a_n)\}$ in  $\ell_1$.
	
	It is enough to prove that $G=Y^\#$.
	
	Let $Y=\ker (a_n)\ci c_0$, where $(a_n)\in S_{\ell_1}$, $|a_N|>\frac{1}{2}$ and $G=\ker (e_N)\ci \ell_1$. Note that for $(b_n)\in G$, $b_N=0$.
	
	Now $\|(b_n)|_Y\|=\ds\sup_{(x_n)\in B_Y}\left|\sum_{n\in \mathbb{N}} b_nx_n\right|\leq \sum_{n\in \mathbb{N}} |b_n|$. For $m>N$, define $x^m$ by,
	\begin{equation*}
	x_n^m = \begin{cases}
	\text{sign } (b_n) & \text{if } n\neq N,\& ~n\leq m,\\
	-\frac{\sum_{n=1, n\neq N}^{m} \left(a_n\text{sign }(b_n)\right) }{a_N} & \text{if } n=N,\\
	0  & \text{if } n>m.
	\end{cases}
	\end{equation*}
	It is clear that $x^m\in B_Y$ and $(b_n)(x^m)=\ds\sum_{n=1}^{m} |b_n|$.
	
	Now $|(b_n)(x^m)|\ra \ds\sum_{n\in \mathbb{N}}|b_n|$ as $m\ra \iy$ and hence $\|(b_n)|_Y\|=\ds\sum_{n\in \mathbb{N}} |b_n|=\|(b_n)\|_1$.
	
	Let $(b_n)\in Y^\#$ and  $\delta>0$ such that, $\ds\sum_{n\in \mathbb{N}\setminus \{N\}}|a_n|<|a_N|-\delta$.
	
	{\sc Claim~:} If $(x_n)\in B_Y$ then $|x_N|<1-\frac{\delta}{|a_N|}$.
	
	Suppose not, i.e $|x_N|\geq 1-\frac{\delta}{|a_N|} $, hence we have $\ds\sum_{n\in \mathbb{N}\setminus \{N\}}|a_n|\geq \left|\sum_{n\in \mathbb{N}\setminus \{N\}}a_nx_n\right|=|x_Na_N|\geq |a_N|(1-\frac{\delta}{|a_N|})=|a_N|-\delta$ contradiction. Now $\|(b_n)\|_1=\ds\sum_{n\in \mathbb{N}}|b_n|=\|(b_n)|_Y\|=\ds \sup_{(x_n)\in B_Y}\ds\sum_{n\in \mathbb{N}}| b_nx_n|\leq \ds\sum_{n\in \mathbb{N}\setminus \{N\}}|b_n|+(1-\frac{\delta}{|a_N|})|b_N|$ implies $\frac{\delta}{|a_N|}|b_N|\leq 0$ i.e $b_N=0$. Hence $(b_n)\in G$.
	
	This completes the proof.
\end{proof}

\begin{ex}\label{Ex2}
	Let $Y_1=\{(x,y,z):x+y+6z=0\}$ and $Y_2=\{(x,y,0):x,y\in \mathbb{R}\}$ be two subspaces of $X=(\mathbb{R}^3,\|.\|_\iy)$. We know that both $Y_1$ and $Y_2$ have property-$SU$ in $X$. Note that $Y_1$ has property-$SU$ in $X$ and $Y_2$ is an M-summand. It is clear that $Y_1\cap Y_2=\{(x,y,0):x+y=0\}$ (=$Z$ say) where $Z^{\perp}=\{(s,s,r):r, s\in\mb{R}\}$, is not a Chebyshev subspace in $X^*$. Hence $Y_1\cap Y_2$ does not have property-$U$ in $X$.
\end{ex}

We now conclude from the above example that, the assertion in Theorem~\ref{T4} is not sufficient to conclude property-$SU$ for finite codimensional ($> 1$) subspaces of $c_0$. Note that $c_0\cong (\mb{R}^n, \|.\|_\iy)\bigoplus_{\ell_\iy} c_0$ and so its dual is $(\mb{R}^n, \|.\|_1) \bigoplus_{\ell_1} \ell_1$, for any natural number $n$. Hence from Example~\ref{Ex2} it is clear that $ker (\frac{1}{8},\frac{1}{8},\frac{3}{4},0,0,\ldots) \bigcap ker (e_3)$ does not have property-$U$ in $c_0$.

The results in Theorem~\ref{T4}, \cite[Theorem~6.1]{BP} and the subsequent discussion lead to the following conclusion.

\begin{cor}
	Let $(a_n)\in S_{\ell_1}$ then $ker (a_n)$ has property-$SU$ in $c_0$ if and only if $\sup_{n\in\mb{N}} |a_n|>\frac{1}{2}$.
\end{cor}

\begin{proof}
	It remains to prove that the condition is necessary.
	
	Since $c_0$ is an $M$-ideal in its bidual $\ell_\infty$ and $Y=\ker (a_n)$ is an ideal in $c_0$, by \cite[Proposition 2, Pg. 605]{RA} $Y$ is an $1$-complemented subspace. Hence $\ds\sup_{n\in \mathbb{N}}|a_n|\geq \frac{1}{2}$ follows from \cite[Theorem~6.1]{BP}.
	
	If possible assume that, $\sup_{n\in \mb{N}}|a_n|=\frac{1}{2}$ then there exist $N\in \mb{N}$ such that $|a_N|=\frac{1}{2}$.
	
	{\sc Case~1:} When $a_N=\frac{1}{2}$.
	
	Now $d(e_N,Y^\perp)=\inf_{\alpha\in \mathbb{R}}\{|1-\alpha\frac{1}{2}|+|\alpha|\sum_{n\in \mathbb{N}, n\neq N}|a_n|\}=\inf_{\alpha\in \mathbb{R}}\{|1-\alpha\frac{1}{2}|+|\alpha|\frac{1}{2}\}$.
	
	Let $d_\alpha=|1-\alpha\frac{1}{2}|+\frac{1}{2}|\alpha|$.

	{\sc Case~1.a:} Let $0\leq \alpha \leq 2$, $d_\alpha=1-\alpha\frac{1}{2}+\alpha\frac{1}{2}=1$.
	
	{\sc Case~1.b:} Let $\alpha>2$, $d_\alpha=\frac{1}{2}\alpha-1+\frac{1}{2}\alpha=\alpha-1>1$.
	
	{\sc Case~1.c:} Let $\alpha<0$, $d_\alpha=1-\frac{1}{2}\alpha-\frac{1}{2}\alpha=1-\alpha>1$.
	
	Therefore $d(e_N,Y^\perp)=1$ and for any $\alpha\in [0,2]$, $(\alpha. a_n)_{n=1}^\iy$ is the best approximation from $e_N$ to $Y^\perp$.
	
	{\sc Case~2:} When $a_N=-\frac{1}{2}$.
	
	Now $d(e_N,Y^\perp)=\inf_{\alpha\in \mathbb{R}}\{|1+\alpha\frac{1}{2}|+|\alpha|\sum_{n\in \mathbb{N}, n\neq N}|a_n|\}=\inf_{\alpha\in \mathbb{R}}\{|1+\alpha\frac{1}{2}|+|\alpha|\frac{1}{2}\}$.
	
	Let $d_\alpha=|1+\alpha\frac{1}{2}|+\frac{1}{2}|\alpha|$.

	{\sc Case~2.a:} Let $\alpha>0$, $d_\alpha=1+\alpha\frac{1}{2}+\alpha\frac{1}{2}=1+\alpha>1$.
	
	{\sc Case~2.b:} Let $-2\leq \alpha\leq 0$, $d_\alpha=\frac{1}{2}\alpha+1-\frac{1}{2}\alpha=1$.
	
	{\sc Case~2.c:} Let $\alpha<-2$, $d_\alpha=-1-\frac{1}{2}\alpha-\frac{1}{2}\alpha=-1-\alpha>1$.
	
	Therefore $d(e_N,Y^\perp)=1$, $(\alpha. a_n)_{n=1}^\iy$ is the best approximation of $e_N$ from $Y^\perp$, for $\alpha\in [-2,0]$.
	
	Hence in any case $Y^\perp$ can not be a Chebyshev subspace of $\ell_1$.
\end{proof}

\subsection*{Acknowledgements}
The research of the second author is supported by Science and Engineering Research Board, India, Award No. MTR/ 2017/ 000061. Corresponding author would like to thank SERB for their financial support.


\end{document}